\documentclass[11pt,a4paper,leqno]{amsart}

\usepackage{amsmath,amsfonts,amssymb,amsthm}
\usepackage[english]{babel}
\usepackage{xcolor}
\usepackage{datetime}
\usepackage{enumerate}
\usepackage[foot]{amsaddr}
\usepackage{lipsum}
\usepackage{minitoc}
\usepackage{hyperref}
\usepackage{pifont}

\usepackage{eso-pic, rotating}

\newtheorem{theo}{Theorem}

\newtheorem{lem}{Lemma}[section]
\newtheorem{coro}{Corollary}[section]
\newtheorem{propo}{Proposition}[section]
\newtheorem{example}{Example}[section]

\theoremstyle{definition}
\newtheorem{remark}[section]{Remark}

\numberwithin{equation}{section}

\newtheorem{lemmaletra}{Lemma}

\newtheorem{propoletra}{Proposition}

\usepackage{xcolor}

\def\R{\mathbb{R}}
\def\C{\mathbb{C}}
\def\D{\mathbb{D}}

\def\E{\mathbf{E}}
\def\V{\mathbf{V}}
\def\P{\mathbf{P}}

\def\K{\mathcal{K}}

\title{Some Gaussianity criteria for Khinchin families}
\author[V.\,J. Maci\'a]{V\'{\i}ctor J. Maci\'a}
\address[V\'{\i}ctor J. Maci\'a]{Departamento de Matem\'aticas, CUNEF Universidad, Spain.}
\email{victor.macia@cunef.edu}


\begin{document}
	
\maketitle	
\vspace{-3em}
\begin{abstract}In this note we establish simple and verifiable analytical conditions for a power series $f$ in the class $\K$, and its associated Khinchin family, to be Gaussian. We give several moment criteria for {general power series} with non-negative coefficients in $\K$ and also provide several applications of these criteria.
\end{abstract}

\section{Introduction}

To each power series $f(z) = \sum_{n = 0}^{\infty}a_nz^n \in \K$, with non-negative Taylor coefficients and having radius of convergence $R>0$ (see the definition of the class $\K$ below), we can assign a family of discrete random variables $(X_t)_{t \in [0,R)}$ with values in the non-negative integers. For each $t \in (0,R)$ the random variable $X_t$ has mass function
\begin{align*}
\P(X_t = n) = \frac{a_n t^n}{f(t)}\,, \quad \text{ for any integer } n \geq 0\,,
\end{align*}
and we define $X_0 \equiv 0$.
The theory of Khinchin families examines the behavior of this family of random variables, and its normalized version, as $t \uparrow R$, see \cite{CFFM1}, \cite{CFFM2} and \cite{CFFM3} for further details.
\medskip

For each \( t \in (0, R) \), we denote the normalization of \( X_t \) by $\breve{X}_t = {(X_t - \E(X_t))}/{\sqrt{\V(X_t)}}$,
where \( \E \) and \( \V \) represent the expectation and variance, respectively. We say that a power series \( f \in \K \), or its associated Khinchin family, is Gaussian if 
\[
\breve{X}_t \stackrel{d}{\rightarrow} \text{N}(0,1), \quad \text{as } t \uparrow R\,,
\]
where \( \stackrel{d}{\rightarrow} \) denotes convergence in distribution and $\text{N}(0,1)$ means standard normal. Equivalently, by virtue of Lévy's continuity theorem, this holds if and only if we have the following pointwise convergence of the characteristic function:
\[
\lim_{t \uparrow R} \E(e^{i\theta \breve{X}_t}) =e^{-\theta^2/2}, \quad \text{for any } \theta \in \R.
\]

The question of whether a power series in the class $\K$ is Gaussian or not is central in the theory of Khinchin families. Some of the most relevant classes of functions in this theory, such as the class of strongly Gaussian power series and the Hayman functions, are contained in the class of Gaussian power series, see, for instance, \cite{CFFM1} for the definition of these classes and also for a proof of this fact. 
\medskip

This question is significant in its own right as well; classifying the limit distributions of Khinchin families associated with power series having non-negative coefficients (specifically, power series in the class $\K$) is an interesting question.

\medskip
In this note we provide simple and verifiable analytical conditions for a power series $f \in \K$, and its associated Khinchin family, to be Gaussian. 
We establish moment criteria for power series with non-negative coefficients in \( \K \) and present several applications of these results. For instance, we prove that the ordinary generating function (OGF) for the partitions of integers with parts in a set \( A \) of positive density is Gaussian, and that entire functions of genus 0 with negative zeros \( f \in \mathcal{K} \), satisfying \( \lim_{t \to +\infty} \sigma_f^2(t) = +\infty \), are also Gaussian.

\subsection{Asymptotic notations.}
For two function $\alpha$ and $\beta$, we say that $\alpha$ and $\beta$ are asymptotically equivalent, as $t \uparrow R$, denoted $\alpha(t) \sim \beta(t)$, as $t \uparrow R$, if
\begin{align*}
\lim_{t \uparrow R}\frac{\alpha(t)}{\beta(t)} = 1\,.
\end{align*}

Similarly, we say that $\alpha$ and $\beta$ are comparable, as $t \uparrow R$, denoted $\alpha(t) \asymp \beta(t)$, as $t \uparrow R$, if there are constants $C,c>0$ such that 
\begin{align*}
c\beta(t) \leq \alpha(t) \leq C\beta(t)\,,\quad \text{ as }t\uparrow R.
\end{align*}

\subsection*{Acknowledgements.} The author extends his gratitude to Prof. José L. Fernández for his invaluable feedback on an earlier draft of this paper and for the many insightful discussions shared over the years.

 \section{Khinchin families}

The fundamental aspects of the theory of Khinchin families required for the proofs in this note are introduced in this section. For a quite detailed presentation see \cite{tesis} and also \cite{CFFM1}, \cite{CFFM2} and \cite{CFFM3}.
\medskip

We denote by $\K$ the class of non-constant power series
$$f(z)=\sum_{n=0}^\infty a_n z^n$$  with positive radius of convergence $R>0$,  which have non-negative Taylor coefficients and such that  $a_0>0$. Since $f \in \K$ is non-constant, at least one coefficient other than $a_0$ is positive.
\medskip

The \textit{Khinchin  family} of  such a  power series $f \in \K$ with radius of convergence $R>0$ is the family of random variables $(X_t)_{t \in [0,R)}$ with values in $\{0, 1, \ldots\}$ and with mass functions given by
$$
\P(X_t=n)=\frac{a_n t^n}{f(t)}\, , \quad \mbox{for each $n \ge 0$ and $t \in (0,R)$}\, .$$  Notice that $f(t)>0$ for each $t \in[0,R)$. We define $X_0\equiv 0$. 
Formally, the definition of $X_0$ is consistent with the general expression for $t \in (0,R)$, with the convention that $0^0 = 1$, meaning that $\P(X_0 = 0) = 1$.
\subsection{Mean and variance functions.} \label{section: mean and variance}

For the mean and variance of $X_t$, we reserve the notation
$m_f(t)=\E(X_t)$ and $\sigma_f^2(t)=\V(X_t)$, for $t \in [0,R)$. In terms of the power series $f \in \K$, the  mean and the variance of $X_t$ may be  written as
\begin{align}\label{eq:m and sigma in terms of f}
	m_f(t)=\frac{t f^\prime(t)}{f(t)} = t\frac{d}{dt}\ln(f(t)), \qquad \sigma_f^2(t)=t m_f^\prime(t)\, , \quad \mbox{for $t \in [0,R)$}\,.
\end{align}

\subsection{Normalization and characteristic functions.}

The characteristic function of $X_t$ may be written in terms of the power series $f$ as:
$$\E(e^{i \theta X_t})=\sum_{n = 0}^{\infty}e^{i\theta n}\frac{a_n t^n}{f(t)} = \frac{f(te^{i \theta})}{f(t)}\, , \quad \mbox{for $t\in (0,R)$ and $\theta \in \R$}\,, $$
while for its normalized version $\breve{X}_t$ we have:
$$\E(e^{i \theta \breve{X}_t})=\E(e^{i \theta X_t/\sigma_f(t)}) e^{-i \theta m_f(t)/\sigma_f(t)}\, , \quad \mbox{for $t\in (0,R)$ and $\theta \in \R$}\,.$$

\subsection{Gaussian power series.} For a power series $f \in \K$, with radius of convergence $R > 0$, we say that $f$, or its associated Khinchin family $(X_t)$, is Gaussian if 
	\begin{align*}
	\breve{X}_t \stackrel{d}{\rightarrow} \text{N}(0,1), \quad \text{ as } t\uparrow R\,,
\end{align*}
or, equivalently, if we have the pointwise limit
\begin{align*}
	\lim_{t \uparrow R}\E(e^{i\theta\breve{X}_t}) = e^{-\theta^2/2}, \quad  \text{ for any } \theta \in \R\,.
\end{align*}

The equivalence above is given by Lévy's continuity theorem; see, for instance, \cite{Durrett}. Recall that $e^{-\theta^2/2}$ is the characteristic function of the standard normal distribution. Here we also use the fact that the distribution function of a standard normal random variable is continuous everywhere.

\subsection{Fulcrum of a power series.}\label{section:auxiliary F} A power series $f$ in $\K$ does not vanish on the real interval $[0,R)$, and so, it does not vanish in a simply connected region containing that interval.  We may consider $\log(f)$, a branch of the logarithm of $f$ which is real on $[0,R)$, and the  function $F$, called the \textit{fulcrum} of $f$, which is  defined and holomorphic in a region containing $(-\infty, \ln R)$ and it is given by
$$F(z)=\log( f(e^z))\, .$$

In terms of the fulcrum $F$ of $f$, the mean and variance functions of the Khinchin family associated to $f$ may be expressed as follows:
\begin{align*}
	m_f(e^s) = F^{\prime}(s) \quad \text{ and } \quad \sigma_f^2(e^s) = F^{\prime\prime}(s), \quad \text{ for } s < \ln(R).
\end{align*}

If $f$ does not vanish anywhere in the disk $\D(0,R)$, then the fulcrum $F(z)$ is defined everywhere in the whole half plane $\Re(z)< \ln R$.

\subsubsection{Zero-free region}

The following proposition gives the zero-free region of $f$ in terms of the variance function of the Khinchin family.
\begin{propoletra}\label{prop:ceros de funcion f de K}
	Let $f\in \mathcal{K}$ have radius of convergence $R>0$.  If for some $t\in [0,R)$ and some $\theta \in [-\pi, \pi]$, we have $f(te^{i \theta})=0$,  then
	$$
	|\theta|\sigma_f(t)\ge \frac{\pi}{2}\, .
	$$
	Thus, $f\in\mathcal{K}$ does not vanish in
	$$
	\Omega_f=\Big\{z=t e^{i \theta}: t \in [0,R) \text{ and } |\theta|< \frac{\pi}{2\sigma_f(t)}   \Big\}\,.
	$$
\end{propoletra}
For a proof of this result, see, for instance, \cite{CFFM3}.

\subsection{Characteristic function of the family $(\breve{X}_t)$ and the fulcrum.}

	 Let $f \in \K$ be a power series with radius of convergence $R > 0$ and Khinchin family $(X_t)$. Fix $t = e^s$, for $s<\ln(R)$. Proposition \ref{prop:ceros de funcion f de K} states that the characteristic function of $\breve{X}_t$ is never zero for $|\theta| < \pi/2$. 
\smallskip

For any fixed $s<\ln(R)$ this bound also implies that the function
$
F(s+i\theta/\sigma_f(t))
$
is well defined for any $|\theta|<\pi/2$. This follows from the fact that $f(e^se^{i\theta/\sigma_f(t)}) = f(e^s)\E(e^{i\theta X_t/\sigma_f(t)})$ is never zero for any fixed $s < \ln(R)$ and any $|\theta| < \pi/2$, see Proposition \ref{prop:ceros de funcion f de K}. From the previous discussion we obtain the following lemma.
\medskip

\begin{lem}\label{lemma: equality_fulcrum_characteristic} Let $f \in \K$ a power series with radius of convergence $R>0$.
For $t = e^s$, with $s < \ln(R)$ fixed, and any $|\theta| < \pi/2$, we have the following relations
\begin{align}\label{eq: characteristica_fulcrum}
	\E(e^{i\theta \breve{X}_t}) &= \exp\left({F(s+i\theta/\sigma_f(t))-F(s)-\frac{F^{\prime}(s)}{F^{\prime\prime}(s)^{1/2}}i\theta}\right)\,, \\
\intertext{and} \label{eq: log_characteristic_fulcrum}
	\ln(\E(e^{i\theta\breve{X}_t})) &= F(s+i\theta/\sigma_f(t))-F(s)-\frac{F^{\prime}(s)}{F^{\prime\prime}(s)^{1/2}}i\theta\,.
\end{align}
For any fixed $t = e^s \in (0,R)$ the functions given by \eqref{eq: characteristica_fulcrum} and \eqref{eq: log_characteristic_fulcrum} are $C^{\infty}$ with respect to $\theta$.
\end{lem}

\subsection{Expressions for the moments of $(\breve{X}_t)$ in terms of the fulcrum.} The next lemma relates the moments of the normalized Khinchin family $(\breve{X}_t)$ and the derivatives of the fulcrum of a power series $f \in \K$.

\begin{theo}\label{lemma: formula_faa_di_bruno_Fulcrum_momentos_tipificados} Let $f \in \K$ be a power series with radius of convergence $R > 0$ and denote $(X_t)$ the Khinchin family associated to $f$, then, for any $t = e^s \in (0,R)$ and any $n \geq 3$, we have
	\begin{align}\label{eq: moments_normalized_Khinchin_fulcrum}
		\E(\breve{X}_t^n) = \sum_{2m_2+\dots+nm_n = n}\frac{n!}{m_2!\dots m_n!}\prod_{j = 1}^{n}\left(\frac{F^{(j)}(s)}{F^{\prime\prime}(s)^{j/2}}\frac{1}{j!} \right)^{m_j}\,,
	\end{align}
	and 
	\begin{align}\label{eq: quotient_derivatives_fulcrum_terms_moments_normalized}
		\frac{F^{(n)}(s)}{F^{\prime\prime}(s)^{n/2}} = \sum_{2m_2+\dots +nm_n = n}\frac{n! (l_n-1)!(-1)^{l_n+1}}{m_2!\dots m_n!}\prod_{j = 1}^{n}\left(\frac{\E(\breve{X}_t^j)}{j!}\right)^{m_j}\,,
	\end{align}
	here $l_n = m_2+\dots+m_n$. 
\end{theo}
\begin{proof}We compute the derivatives of any order of \eqref{eq: characteristica_fulcrum} and \eqref{eq: log_characteristic_fulcrum} at $\theta = 0$. 
\medskip

Assume that $f = e^g$, where $g$ is a $C^{\infty}$ function such that $g(0) = 0$ and $g^{\prime}(0) = 0$.
A direct application of Faà di Bruno's formula to $f = e^g$, see, for instance, \cite[p. 217]{Warren}, gives that for any $n \geq 1$ we have
\smallskip
\begin{align}\label{eq: FaaDiBruno_exponentials}
\begin{split}
f^{(n)}(0) &= \sum_{m_1+\dots+nm_n = n}\frac{n!}{m_1!\dots m_n!}\prod_{j = 1}^{n}\left(\frac{g^{(j)}(0)}{j!}\right)^{m_j} \\[.15 cm]
&= \sum_{2m_2+\dots+nm_n = n} \frac{n!}{m_1!\dots m_n!}\prod_{j = 1}^{n}\left(\frac{g^{(j)}(0)}{j!}\right)^{m_j}\,.
\end{split}
\end{align}
\smallskip

For any fixed $t \in (0,R)$ and any $|\theta|<\pi/2$, the characteristic function of $\breve{X}_t$ can be written as $e^g$, for some $C^{\infty}$ function $g$ such that $g(0) = 0$ and $g^{\prime}(0) = 0$, see equation (\ref{eq: characteristica_fulcrum}). 
\smallskip

Fix $t = e^s \in (0,R)$, and define the function
\begin{align*}
g(\theta) = F(s+i\theta/\sigma_f(t))-F(s)-\frac{F^{\prime}(s)}{F^{\prime\prime}(s)^{1/2}}i\theta, \quad \text{ for any } |\theta|<\pi/2\,.
\end{align*}
Observe that $g(0) = 0$ and $g^{\prime}(0) = 0$. Combining equation \eqref{eq: FaaDiBruno_exponentials} with Lemma \ref{lemma: equality_fulcrum_characteristic} we obtain that 
	\begin{align*}
	\E(\breve{X}_t^n) = \sum_{2m_2+\dots+nm_n = n}\frac{n!}{m_2!\dots m_n!}\prod_{j = 1}^{n}\left(\frac{F^{(j)}(s)}{F^{\prime\prime}(s)^{j/2}}\frac{1}{j!} \right)^{m_j}\,.
\end{align*}
\medskip

For the other relation: fix $t = e^s \in (0,R)$ and $n \geq 1$. By applying Faà di Bruno formula to \eqref{eq: log_characteristic_fulcrum}, we obtain the $n$-th derivative of \eqref{eq: log_characteristic_fulcrum} at $\theta = 0$:
	\begin{align}\label{eq: quotient_derivatives_fulcrum_terms_moments_normalized}
	\frac{F^{(n)}(s)}{F^{\prime\prime}(s)^{n/2}} = \sum_{2m_2+\dots +nm_n = n}\frac{n! (l_n-1)!(-1)^{l_n+1}}{m_2!\dots m_n!}\prod_{j = 1}^{n}\left(\frac{\E(\breve{X}_t^j)}{j!}\right)^{m_j}\,,
\end{align}
where $l_n = m_2+\dots+m_n$. 

\end{proof}

Let $f \in \K$ a power series with radius of convergence $R>0$. For $n=3$, applying Lemma \ref{eq: moments_normalized_Khinchin_fulcrum}, we obtain that
\begin{align*}
\E(\breve{X}_t^3) = \frac{F^{\prime\prime\prime}(s)}{F^{\prime\prime}(s)^{3/2}},  \quad \text{ for any } t = e^s \in (0,R)\,.
\end{align*}
This tell us that the 3rd central moment of the family $(X_t)$ is given by the 3rd derivative of the fulcrum. However, this is not a general fact; for $n = 4$ we have that 
\begin{align*}
\E(\breve{X}_t^4) = \frac{F^{(4)}(s)}{F^{\prime\prime}(s)^2}+\frac{4!}{2!(2!)^2},\quad \text{ for any } t = e^s \in (0,R)\,,
\end{align*}
and the central moments of $({X}_t)$ cannot be expressed in terms of a single derivative of the fulcrum.
\medskip

The moments of the normal distribution are included in the sum \eqref{eq: moments_normalized_Khinchin_fulcrum}. 
\begin{coro} Let $f \in \K$ a power series with radius of convergence $R>0$ and let $(X_t)$ be its Khinchin family. Fix $t =e^s\in (0,R)$, then
\vspace{.2 cm}
\begin{itemize}
\renewcommand{\labelitemi}{\raisebox{0.3ex}{\scalebox{0.6}{$\blacksquare$}}}
\item For even $n = 2k$ with $k \geq 2$ we have 
	\begin{align*}
	\E(\breve{X}_t^n)-\frac{(2k)!}{k! \, (2!)^k}= \sum_{\substack{2m_2+\dots+nm_n = n \\ m_2 \neq k}}\frac{n!}{m_2!\dots m_n!}\prod_{j = 1}^{n}\left(\frac{F^{(j)}(s)}{F^{\prime\prime}(s)^{j/2}}\frac{1}{j!} \right)^{m_j}\,,
\end{align*}\\[-.3 cm]
\item For odd $n=2k+1$ with $k \geq 1$ we have 
\begin{align*}
		\E(\breve{X}_t^n) = \sum_{2m_2+\dots+nm_n = n}\frac{n!}{m_2!\dots m_n!}\prod_{j = 1}^{n}\left(\frac{F^{(j)}(s)}{F^{\prime\prime}(s)^{j/2}}\frac{1}{j!} \right)^{m_j}\,,
\end{align*}
\end{itemize}
and all the terms in the sums at the right-hand side of the previous expressions are non-constant (with respect to $s$).
\end{coro}

\begin{proof}
	Fix $n = 2k$, where $k \geq 2$ is an integer. Formula (\ref{eq: moments_normalized_Khinchin_fulcrum}) can be written as:
	\begin{align*}
		\E(\breve{X}_t^n) = \sum_{\substack{2m_2+\dots+nm_n = n \\ m_2 \neq k}}\frac{n!}{m_2!\dots m_n!}\prod_{j = 1}^{n}\left(\frac{F^{(j)}(s)}{F^{\prime\prime}(s)^{j/2}}\frac{1}{j!} \right)^{m_j}+\frac{(2k)!}{k! \, (2!)^k}\,,
	\end{align*}
	and therefore 
		\begin{align}\label{eq: even_moments_proof}
		\E(\breve{X}_t^n) -\frac{(2k)!}{k! \, (2!)^k}= \sum_{\substack{2m_2+\dots+nm_n = n \\ m_2 \neq k}}\frac{n!}{m_2!\dots m_n!}\prod_{j = 1}^{n}\left(\frac{F^{(j)}(s)}{F^{\prime\prime}(s)^{j/2}}\frac{1}{j!} \right)^{m_j}\,.
	\end{align}
	
	Assume that $n =2k$, for $k \geq 2$, is even. For any list of non-negative integers $m_2,m_3,\dots,m_n$ such that $2m_2+3m_3+\dots+nm_n = n$, with $m_2 \neq k$, there always exists some $m_j >0$, with $3\leq j\leq n$. Therefore each of the terms in the sum \eqref{eq: even_moments_proof} includes a non-trivial (i.e. $j \neq 2$) quotient of derivatives of the fulcrum.

	\medskip
	Assume now that $n$ is odd. For any list of non-negative integers $m_2,m_3,\dots,m_n$ such that $2m_2+3m_3+\dots+nm_n = n$, there always exists some odd $j$ with $3\leq j\leq n$ such that $m_j>0$, otherwise $n$ would be even. Therefore, each term in the sum 
	\begin{align*}
		\E(\breve{X}_t^n) = \sum_{2m_2+\dots+nm_n = n}\frac{n!}{m_2!\dots m_n!}\prod_{j = 1}^{n}\left(\frac{F^{(j)}(s)}{F^{\prime\prime}(s)^{j/2}}\frac{1}{j!} \right)^{m_j}\,,
	\end{align*}
	includes a non-trivial (i.e. $j\neq2$) quotient of derivatives of the fulcrum.
\end{proof}

\begin{remark}
For comparison and later use we collect here the integer moments of a standard normal random variable: 
	\begin{align*}
		\E(Z^n) = 0, \quad \text{ for odd } n \geq 1,
	\end{align*}
	and 
	\begin{align*}
		\E(Z^{2k}) = \frac{(2k)!}{k! \, (2!)^{k}}, \quad \text{ for any } k \geq 0.
	\end{align*}
See, for instance, \cite[p. 114]{Durrett}
%
\end{remark}

\subsection{Riordan formula.}
There is an equivalent form of Faà di Bruno's formula known as Riordan's formula; see, for instance, \cite[p. 219]{Warren}. This equivalent formula makes use of Bell polynomials.
	The $n$-th Bell polynomial is given by 
	\begin{align}\label{eq: Bellpolynomials}
		B_n(x_1,\dots,x_n) = n! \sum_{j_1+2j_2+\dots+nj_n = n}\prod_{l=1}^{n}\frac{x_l^{j_l}}{(l!)^{j_l}j_l!}\,,
	\end{align}
	see, for instance, \cite[p. 173]{Riordan} or \cite[pp. 218-219]{Warren}. 
	\medskip
	
By applying Riordan’s formula, see \cite[p. 219]{Warren}, we derive an equivalent, and significantly more compact, expression for \eqref{eq: moments_normalized_Khinchin_fulcrum}:
	\begin{align}\label{eq: Riordan_formula_aplicadad_tipificadaKhinchin}
		\small
		\begin{split}
			\E(\breve{X}_t^n)i^n &=  \sum_{k=0}^{n}B_{n,k}\left(0,i^2,i^3\frac{F^{\prime\prime\prime}(s)}{F^{\prime\prime}(s)^{3/2}},\dots,i^{n-k+1}\frac{F^{(n-k+1)}(s)}{F^{\prime\prime}(s)^{(n-k+1)/2}}\right) \\[.2 cm]
			&= B_n\left(0,i^2,i^3\frac{F^{\prime\prime\prime}(s)}{F^{\prime\prime}(s)^{3/2}},\dots,i^{n}\frac{F^{(n)}(s)}{F^{\prime\prime}(s)^{n/2}}\right) = i^n B_n\left(0,1,\frac{F^{\prime\prime\prime}(s)}{F^{\prime\prime}(s)^{3/2}},\dots,\frac{F^{(n)}(s)}{F^{\prime\prime}(s)^{n/2}}\right).
		\end{split}
	\end{align}
	Here $B_{n,k}(x_1,\dots,x_{n-k+1})$ denotes the $n$-th incomplete exponential Bell polynomial, and  $B_n(x_1,\dots,x_n)$ denotes the $n$-th Bell polynomial, with $n \geq 1$ being an integer.
	\medskip
	
	In the expression above, i.e equation \eqref{eq: Riordan_formula_aplicadad_tipificadaKhinchin}, we make use of the following property of the Bell polynomials: for any complex number $a \in \C$ we have that
	\begin{align*}
		(\dagger) \quad 	B_n(ax_1,a^2x_2,\dots,a^n x_n) = a^nB_n(x_1,\dots,x_n).
	\end{align*}
This equality follows directly from the definition of the Bell polynomials in \eqref{eq: Bellpolynomials}.
	
	\medskip

	Denote by $Z$ a standard normal random variable. Applying Riordan's formula, see \cite[p. 219]{Warren}, to the characteristic function $\E(e^{i\theta Z}) = e^{-\theta^2/2}$, at $\theta = 0$, we find that 
	\begin{align}\label{eq: moments_standard_normal_Bell}
		\E(Z^n)i^n = \sum_{k = 0}^{n}B_{n,k}(0,-1,0, \dots,0) = B_n(0,-1,0,\dots,0) = i^nB_n(0,1,0,\dots,0)\,,
	\end{align}
Here, \( B_{n,k}(x_1,\dots,x_{n-k+1}) \) denotes the $n$-th incomplete exponential Bell polynomial, and \( B_n(x_1,\dots,x_n) \) represents the $n$-th Bell polynomial, both for integers $n \geq 1$. The equality on the right-hand side follows from the property $(\dagger)$ of the Bell polynomials.
\smallskip

In summary, formula \eqref{eq: moments_standard_normal_Bell} gives that the $n$-th moment of the standard normal distribution can be expressed in terms of the $n$-th Bell polynomial. Specifically, we have
\begin{align}\label{eq: moments_normal_Bernoullipoly}
\E(Z^n) = B_n(0, 1, 0, \dots, 0), \quad \text{for any } n \geq 1\,,
\end{align}
Furthermore, using equation \eqref{eq: Riordan_formula_aplicadad_tipificadaKhinchin}, we can, similarly, express the moments of $(\breve{X}_t)$ as:

\begin{align}\label{eq: formula_momentos_tipificada_Khinchin}
\E(\breve{X}_t^n) = B_n\left(0, 1, \frac{F^{\prime\prime\prime}(s)}{F^{\prime\prime}(s)^{3/2}}, \dots, \frac{F^{(n)}(s)}{F^{\prime\prime}(s)^{n/2}}\right), \quad \text{for any } n \geq 1\,.
\end{align}

We can also express the $n$-th central moment of the Khinchin family $(X_t)$ in terms of the $n$-th Bell polynomial. Multiplying both sides of \eqref{eq: formula_momentos_tipificada_Khinchin} by $F^{\prime\prime}(s)^{n/2}$ we obtain that 
\begin{align*}
\E((X_t-m_f(t))^{n}) = F^{\prime\prime}(s)^{n/2}B_n\left(0, 1, \frac{F^{\prime\prime\prime}(s)}{F^{\prime\prime}(s)^{3/2}}, \dots, \frac{F^{(n)}(s)}{F^{\prime\prime}(s)^{n/2}}\right).
\end{align*}
Next, we make use of property $(\dagger)$ to stablish that
\begin{align*}
\E((X_t-m_f(t))^{n}) = B_n\left(0, F^{\prime\prime}(s), {F^{\prime\prime\prime}(s)}, \dots, {F^{(n)}(s)}\right).
\end{align*}
Recall that here we parametrize $t = e^s$, for $s < \ln(R)$.

\begin{theo}\label{lemma: convergence_moments_fulcrum}Let $f \in \K$ be a power series with radius of convergence $R>0$. Denote by $Z$ a standard normal random variable. Fix an integer $n\geq3$, then 
	\begin{align*}
		\lim_{t\uparrow R}\E(\breve{X}_t^n) = \E(Z^n)\,, \quad \text{ for any } 3\leq j\leq n,
	\end{align*}
	if and only if 
	\begin{align*}
		\lim_{s\uparrow\ln(R)}\frac{F^{(j)}(s)}{F^{\prime\prime}(s)^{j/2}} = 0\,,\quad \text{ for any } 3\leq j\leq n. 
	\end{align*}
\end{theo}
\begin{proof} First observe that, for any $j \geq 3$, the term ${F^{(j)}(s)}/{F^{\prime\prime}(s)^{j/2}}$ always appears in the $j$-th moment of $\breve{X}_t$, see equation \eqref{eq: moments_normalized_Khinchin_fulcrum} for further details.
\medskip

We fix some \(n \geq 3\) and impose the condition:
\[
\lim_{t \uparrow R} \E(\breve{X}_t^n) = \E(Z^n)\,, \text{ for any } 3\leq j\leq n\,,
\]
this in turn implies that
\[
\lim_{s \uparrow \ln(R)} \frac{F^{(j)}(s)}{F^{\prime\prime}(s)^{j/2}} = 0 \quad \text{for all} \quad 3 \leq j \leq n.
\]

Conversely, to ensure the convergence of these moments, we require the convergence of these quotients of derivatives for every \(3 \leq j \leq n\).
\medskip

\end{proof}

\section{Moment criteria for Gaussianity}

Let $Z$ be a standard normal random variable. The standard normal distribution is determined by its moments, meaning that: For any sequence of random variables $(X_n)_{n \geq 1}$ such that
\begin{align*}
\lim_{n \rightarrow +\infty}\E(X_n^m) = \E(Z^m)\,, \quad \text{ for any } m \geq 1,
\end{align*}
we have $X_n \stackrel{d}{\rightarrow} Z$, as $n \rightarrow +\infty$, see, for instance, Billingsley, \cite[Section 30]{Billingsley}. Combining \cite[Theorem 30.2]{Billingsley} and Theorem \ref{lemma: convergence_moments_fulcrum} we have
%
\begin{theo}\label{thm: gaussianity all derivatives}
	Let $f \in \K$ be a power series with radius of convergence $R>0$. Assume that 
	$$\lim_{s\uparrow \ln R} \frac{F^{(k)}(s)}{F^{\prime\prime}(s)^{k/2}}=0\,, \quad \mbox{for every $k \ge 3$}\,,$$
	then $f$ is gaussian.
\end{theo}
	\smallskip


The significance of the previous result lies in the fact that all the information about the moments of the family $(\breve{X}_t)$ depends solely on the values of $t$ on the positive real axis. Consequently, the method of moments allows us to prove that $f$ is Gaussian using only the restriction of $f$ to the positive real axis. For further comparisons, see Theorem 3.2 in \cite[p. 864]{CFFM1}, where it is assumed that $f$ has no zeros in its disk of convergence and also that a uniform condition on the third derivative of the fulcrum holds around the real axis.
\medskip

We now provide conditions, in terms of the derivatives of the fulcrum, for the normalized Khinchin family $(\breve{X}_t)$ to have bounded integer moments. 
\begin{lem}\label{lemma: finitemomentsiff}Let $f \in \K$ be a power series with radius of convergence $R>0$. For any integer $n \geq 3$ we have that
	\begin{align*}
		\limsup_{s \uparrow \ln(R)}\frac{|F^{(j)}(s)|}{F^{\prime\prime}(s)^{j/2}}<+\infty\,, \quad \text{ for any } 3 \leq j\leq n\,,
	\end{align*}
	if and only if 
	\begin{align*}
		\limsup_{t\uparrow R}|\E(\breve{X}_t^k)|<+\infty\,, \quad \text{ for any } 3 \leq k \leq n\,.
	\end{align*}
\end{lem}
\begin{proof} Using formula \eqref{eq: moments_normalized_Khinchin_fulcrum} we find that, for any $3 \leq k \leq n$ and $t = e^s \in (0,R)$, we have
	\begin{align}\label{eq: finite_quotients_derivative_j}
		\E(\breve{X}_t^k)= B_k\left(0,1,\frac{F^{\prime\prime\prime}(s)}{F^{\prime\prime}(s)^{3/2}},\dots,\frac{F^{(k)}(s)}{F^{\prime\prime}(s)^{k/2}}\right)
	\end{align}
	here $B_k(x_1,\dots,x_n)$ is $k$-th Bell polynomial. The previous equality gives that 
	\begin{align*}
		|\E(\breve{X}_t^k)| \leq B_k\left(0,1,\frac{|F^{\prime\prime\prime}(s)|}{F^{\prime\prime}(s)^{3/2}},\dots,\frac{|F^{(k)}(s)|}{F^{\prime\prime}(s)^{k/2}}\right).
	\end{align*}
	Here we use that the Bell polynomials have non-negative coefficients. Using \eqref{eq: finite_quotients_derivative_j} we conclude that 
	\begin{align*}
		\limsup_{t\uparrow R}|\E(\breve{X}_t^k)|<+\infty\,, \quad \text{ for any } 3 \leq k \leq n\,.
	\end{align*}
	
	Now for the other implication, using formula \eqref{eq: quotient_derivatives_fulcrum_terms_moments_normalized}, we obtain that for any $3\leq k\leq n$ there exists a constant $C_k>0$ such that 
	\begin{align*}
		\frac{|F^{(k)}(s)|}{F^{\prime\prime}(s)^{k/2}} \leq C_k \cdot B_{k}(0,1,|\E(\breve{X}_t^3)|,\dots,|\E(\breve{X}_t^{k})|)\,.
	\end{align*}
	Here $B_k(x_1,\dots,x_k)$ denotes the $k$-th Bell polynomial. 
\end{proof}

\subsection{A partial converse of Theorem \ref{thm: gaussianity all derivatives}.}
Combining the concept of uniform integrability, see, for instance, \cite{Billingsley2}[p. 30], with Lemma \ref{lemma: finitemomentsiff} we obtain a \emph{partial converse} of Theorem \ref{thm: gaussianity all derivatives}.
\begin{theo}\label{thm: bounded_gaussian_thenzero} Let $f \in \K$ be a power series with radius of convergence $R > 0$. Assume that $f$ is Gaussian and also that for some even integer $n \geq 4$ we have 
	\begin{align*}
		\limsup_{s \uparrow \ln(R)}\frac{|F^{(j)}(s)|}{F^{\prime\prime}(s)^{j/2}}<+\infty\,, \quad \text{ for any } 3 \leq j\leq n\,,
	\end{align*}
	then 
	\begin{align*}
		\lim_{s \uparrow \ln(R)}\frac{F^{(j)}(s)}{F^{\prime\prime}(s)^{j/2}} = 0\,, \quad \text{ for any } 3 \leq j <n\,. 
	\end{align*}
\end{theo}
\begin{proof} Our hypothesis combined with Lemma \ref{lemma: finitemomentsiff} gives that 
	\begin{align*}
		\limsup_{t\uparrow R}|\E(\breve{X}_t^k)|<+\infty\,, \quad \text{ for any } 3 \leq k \leq n\,.
	\end{align*}
	
	Because $n$ is even we have that the family $(\breve{X}_t^m)$ is uniformly integrable for any integer $3 \leq m < n$ and therefore using that $f$ is Gaussian we find that 
	\begin{align}\label{eq: moment_convergence_converse}
		\lim_{t \uparrow R}\E(\breve{X}_t^m) = \E(Z^m), \quad \text{ for any integer } 3\leq m <n,
	\end{align}
	here $Z$ denotes a standard normal random variable. Now applying Theorem \ref{lemma: formula_faa_di_bruno_Fulcrum_momentos_tipificados}, see also Theorem \ref{lemma: convergence_moments_fulcrum}, we obtain that \eqref{eq: moment_convergence_converse} is equivalent to 
	\begin{align*}
		\lim_{s \uparrow \ln(R)}\frac{F^{(j)}(s)}{F^{\prime\prime}(s)^{j/2}} = 0\,, \quad \text{ for any } 3 \leq j <n\,. 
	\end{align*}
\end{proof}

\subsection{A partial characterization.}
Combining Theorem \ref{thm: gaussianity all derivatives} and Theorem \ref{thm: bounded_gaussian_thenzero} we obtain a \emph{partial characterization} of gaussianity in terms of quotients of derivatives of the fulcrum. 
\begin{theo}\label{thm: partial_charac_gaussianity} Let $f \in \K$ be a power series with radius of convergence $R>0$. Assume that 
	\begin{align*}
		\limsup_{s\uparrow\ln(R)}\frac{|F^{(k)}(s)|}{F^{\prime\prime}(s)^{k/2}}<+\infty\,, \quad \text{ for any } k\geq 3\,,  
	\end{align*}
	then 
	\begin{align*}
		f \text{ is Gaussian } \quad \text{ if and only if } \quad \lim_{s\uparrow\ln(R)}\frac{F^{(k)}(s)}{F^{\prime\prime}(s)^{k/2}} = 0\,, \quad \text{ for any } k \geq 3\,.
	\end{align*}
\end{theo}
\smallskip

\section{Some applications}

In this section, we illustrate the use of these gaussianity criterion through various applications. By applying the previously presented results to the following specific cases, we show how our criterion can be effectively employed in practical scenarios.

\subsection{Partitions of integers} In this subsection we prove that several ordinary generating function of partitions of integers with restrictions on its parts are Gaussian. \\

As a first step, we state and prove the following lemma.

\begin{lem}\label{lemma: Euler_summation} Let $g \in C^{1}([0,+\infty))$ be a function such that $g(0) = 0$ and $\displaystyle\lim_{t \rightarrow+\infty}g(t)~=~0$, then
	\begin{align}
		\left| s\sum_{j = 1}^{\infty}g(js)-\int_{0}^{\infty}g(x)dx \right| &\leq s\int_{0}^{\infty}|g^{\prime}(x)|dx\,, \label{eq: euler_ineq_withs}
	\end{align}
and 
	\begin{align}\label{eq: Euler_summation_ineq_2}
	\left|s\sum_{j = 1}^{\infty}g(js)-\int_{0}^{\infty}g(x)dx\right| \leq \int_{0}^{\infty}x|g^{\prime}(x)|dx\,.
\end{align}
	
\end{lem}
\begin{proof} Let $f$ be a $C^1([0,+\infty))$ function such that $f(0) = 0$ and $\lim_{t \rightarrow +\infty}f(t) = 0$. Euler's summation formula of order $1$, see, for instance, \cite[p. 41]{DeBruijn}  , gives that  
	\begin{align*}
		\sum_{j = 0}^{\infty}f(j) &= \int_{0}^{\infty}f(t)dt +\int_{0}^{\infty}\tilde{B_1}(t)f^{\prime}(t)dt \\
		&= \int_{0}^{\infty}f(t)dt +\int_{0}^{\infty}\{t\}f^{\prime}(t)dt\,,
	\end{align*}
	Here $\tilde{B}_1(t) = \{t\}-1/2$. Observe that $\int_{0}^{\infty}f^{\prime}(t)dt = 0$.
	\medskip
	
	Now we apply this formula to $f(x) = g(xs)$ to obtain that 
	\begin{align*}
		\sum_{j = 1}^{\infty}g(js)-\frac{1}{s}\int_{0}^{\infty}g(x)dx = \int_{0}^{\infty}\left\{\frac{x}{s}\right\}g^{\prime}(x)dx\,,
	\end{align*}
	and therefore 
	\begin{align}\label{eq: Euler_summation}
		s\sum_{j = 1}^{\infty}g(js)-\int_{0}^{\infty}g(x)dx = s\int_{0}^{\infty}\left\{\frac{x}{s}\right\}g^{\prime}(x)dx\,.
	\end{align}
	\medskip
	
	From equality \eqref{eq: Euler_summation} we obtain the inequality:
	\begin{align*}
		\left| s\sum_{j = 1}^{\infty}g(js)-\int_{0}^{\infty}g(x)dx \right| &\leq s\int_{0}^{\infty}|g^{\prime}(x)|dx\,,.
	\end{align*}
	
	The second inequality follows directly from Euler's summation formula. This concludes the proof of the Lemma.
\end{proof}

\subsubsection{OGF of the partitions of integers into $p$-th powers.} Now we prove that the integer moments of the normalized Khinchin family of the OGF of the partitions into $p$-th powers converge to the integers moments of a standard normal random variable. In particular this implies that the OGF of the partitions into $p$-th powers is Gaussian. 
\medskip

The OGF of the partitions of integers into $p$-th powers is given by the infinite product
\begin{align*}
	G_p(z) = \prod_{j=1}^{\infty}\frac{1}{1-z^{j^p}}\,, \quad \text{ for any } |z|<1\,.	
\end{align*}
\medskip

The fulcrum of $G_p$ is given by 
\begin{align*}
	F_p(s) = -\sum_{j=1}^{\infty}\ln(1-e^{j^ps})=\sum_{j,k\geq1}\frac{e^{j^pks}}{k}\,, \quad \text{ for any } s<0\,,
\end{align*}
and therefore, for any integer $m \geq 0$, we have
\begin{align}\label{eq: derivatives_fulcrum}
	F^{(m)}_p(-s) = \sum_{j,k\geq1}k^{m-1}j^{pm}{e^{-j^pks}}\,\, \quad \text{ for any }s>0\,.
\end{align}
We want an asymptotic formula for $F_p^{(m)}(s)$, as $s \uparrow 0$. To derive this asymptotic formula, we will make use of Lemma \ref{lemma: Euler_summation}.

\begin{propo}\label{lem: asymptotic_Gp_moments} For any pair of integers $p \geq 1$ and $m \geq 0$ we have that
	\begin{align*}
		F_p^{(m)}(-s) \sim \frac{1}{p}\,\zeta(1+1/p)\, \Gamma(m+1/p)\frac{1}{s^{m+1/p}}\,,\quad \text{ as } s\downarrow0\,.
	\end{align*}
\end{propo} 
\begin{proof} For any $s>0$ we can write 
	\begin{align}\label{eq: fulcrum_as_series}
		s^{pm+1}F_p^{(m)}(-s^p) = \sum_{k = 1}^{\infty}k^{m-1}\left[s\sum_{j = 1}^{\infty}(js)^{pm}e^{-(js)^{p}k}\right]\,,
	\end{align} 
	see equation \eqref{eq: derivatives_fulcrum} above. 
	\medskip
	
	We proceed in two steps. First we will use inequality \eqref{eq: euler_ineq_withs} in Lemma \ref{lemma: Euler_summation} to obtain that
	\begin{align*}
		\lim_{s \downarrow 0}s\sum_{j = 1}^{\infty}(js)^{pm}e^{-(js)^pk} = \frac{1}{p}\frac{1}{k^{1+1/p}}\Gamma(m+1/p)\,, \quad \text{ for any } k \geq 1\,. 
	\end{align*}
	Next we will give a summable upper bound. Combining this upper bound with the pointwise convergence we can apply dominated convergence and therefore obtain the asymptotic result. 
	\medskip 
	
	Fix $k,m$ and $p$. Denote $g(x) = k^{m-1}x^{pm}e^{-x^pk}$. Applying inequality \eqref{eq: euler_ineq_withs} in Lemma \ref{lemma: Euler_summation} we obtain that 
	\begin{align*}
		\left|sk^{m-1}\sum_{j = 1}^{\infty}(js)^{pm}e^{-(js)^pk}-\int_{0}^{\infty}g(x)dx\right| \leq s \int_{0}^{\infty}|g^{\prime}(x)|dx.
	\end{align*}
	\medskip
	
	Observe that 
	\begin{align*}
		\int_{0}^{\infty}g(x)dx = \frac{1}{p}\frac{1}{k^{1+1/p}}\Gamma(m+1/p)\,,
	\end{align*}
	\medskip
	then we conclude that
	\begin{align*}
		\left|sk^{m-1}\sum_{j = 1}^{\infty}(js)^{pm}e^{-(js)^pk}-\frac{1}{p}\frac{1}{k^{1+1/p}}\Gamma(m+1/p)\right| \leq s \int_{0}^{\infty}|g^{\prime}(x)|dx.
	\end{align*}
	The integral at the right-hand side of the previous inequality is finite and does not depend on $s$, therefore, for each integer $k \geq 1$, we have
	\begin{align*}
		\lim_{s \downarrow 0}s\sum_{j = 0}^{\infty}(js)^{pm}e^{-(js)^pk} = \frac{1}{p}\frac{1}{k^{1+1/p}}\Gamma(m+1/p)\,.
	\end{align*}
	\medskip
	
	We also have using inequality \eqref{eq: Euler_summation_ineq_2}  that 
	\begin{align}
		\left| s\sum_{j = 1}^{\infty}g(js)-\int_{0}^{\infty}g(x)dx \right| &\leq \int_{0}^{\infty}x|g^{\prime}(x)|dx\,. 
	\end{align}
	for any $s>0$. 
	\medskip
	
	Observe that 
	\begin{align*}
		g^{\prime}(x) = pk^{m-1}x^{pm-1}e^{-kx^p}(m-kx^{p})
	\end{align*}
	therefore, making the change of variables $y = x^pk$, we find that
	\begin{align*}
		\int_{0}^{\infty}x|g^{\prime}(x)|dx = \frac{1}{k^{1+1/p}}\int_{0}^{\infty}y^{m+1/p-1}e^{-y}|m-y|dy\,.
	\end{align*}
	\medskip

	Using this upper bound we conclude that there exists a constant $\tilde{C}_{m,p}>0$, not depending on $k$, such that 
	\begin{align}
		\left| s\sum_{j = 1}^{\infty}g(js)-\int_{0}^{\infty}g(x)dx \right| &\leq \tilde{C}_{m,p} \frac{1}{k^{1+1/p}}\,,
	\end{align}
	for any $s>0$.
	\medskip
	
	%
	%
	%
	%
	%
	
	Now applying dominated convergence it follows that 
	\begin{align*}
		\lim_{s \downarrow 0}s^{pm+1}F_p^{(m)}(-s^p) = \lim_{s \downarrow 0}\sum_{k = 1}^{\infty}sk^{m-1}\sum_{j = 0}^{\infty}(js)^{pm}e^{-(js)^pk} = \frac{1}{p}\zeta(1+1/p)\Gamma(m+1/p)\,.
	\end{align*}
	This proves the claim. 
\end{proof}

A direct application of Proposition \ref{lem: asymptotic_Gp_moments} gives that for any $m \geq 3$ there exists a constant $C_{p,m}>0$ such that 
\begin{align}\label{eq: asymptotic_quotient_fulcrums_G_p}
	\frac{F_p^{(m)}(-s)}{F_p^{\prime\prime}(-s)^{m/2}} \sim C_{p,m} \,\, s^{(m/2-1)/p}\,, \quad \text{ as } s\downarrow0\,,
\end{align}
and therefore, using that $m \geq 3$, and applying Theorem \ref{thm: gaussianity all derivatives}, we conclude that for any $p \geq 1$ the OGF of the partitions of integers into powers of $p$ is Gaussian. See also \cite{CFFM1} for a proof of the Gaussian nature of the OGFs of various types of partitions.

\subsubsection{The MacMahon function: OGF of the plane partitions.} Here we prove that the integer moments of the normalized Khinchin family associated to the OGF of the plane partitions converge to the integer moments of a standard Gaussian. 
\medskip

The ordinary generating function of the plane partitions $\pi(n)$ (also know as the MacMahon function) is given by 
\begin{align*}
M(z) = \prod_{j = 1}^{\infty}\frac{1}{(1-z^j)^j}\,, \quad \text{ for any } |z|<1. 
\end{align*}

The fulcrum of $M(z)$ is given by 
\begin{align*}
F(s) = -\sum_{j = 1}^{\infty}j\ln(1-e^{sj})	= \sum_{j=1}^{\infty}\sum_{k=1}^{\infty}j\frac{e^{(js)k}}{k}\,, \quad \text{ for any } s<0\,,
\end{align*}
and therefore, for any $m \geq 0$, we have
\begin{align*}
F^{(m)}(-s) = \sum_{j=1}^{\infty}\sum_{k=1}^{\infty}j^{m+1}k^{m-1}{e^{-(js)k}}\,, \quad \text{ for any } s>0\,.
\end{align*}

We want an asymptotic formula for $F^{(m)}(s)$, as $s \uparrow 0$. To derive this asymptotic formula, again, we will make use of Lemma \ref{lemma: Euler_summation}.

\begin{propo}\label{propo: asymptotic_MacMahon_moments} For any integer $m \geq 0$ we have that
	\begin{align*}
		F^{(m)}(-s) \sim \zeta(3)\Gamma(m+2)\frac{1}{s^{m+2}}\,,\quad \text{ as } s\downarrow0\,.
	\end{align*}
\end{propo} 
\begin{proof} Fix an integer $m \geq 0$. For any $s>0$ we can write
\begin{align*}
s^{m+2}F^{(m)}(-s) = s\sum_{j=1}^{\infty}\sum_{k=1}^{\infty}(sj)^{m+1}k^{m-1}{e^{-(js)k}}\,, \quad \text{ for any } s>0\,.
\end{align*}
Since this is a double series with non-negative terms, Tonelli’s theorem allows the exchange of summations, a step that will be taken later.
\medskip

Denote $g(x) = k^{m-1}x^{m+1}e^{-kx}$. Applying Lemma \ref{lemma: Euler_summation} we obtain that 
		\begin{align}\label{eq: Euler_sum_MacMahon_ineq}
	\left|s\sum_{j=1}^{\infty}(sj)^{m+1}k^{m-1}{e^{-(js)k}}-\int_{0}^{\infty}g(x)dx\right| \leq s \int_{0}^{\infty}|g^{\prime}(x)|dx\,.
\end{align}
\medskip

Additionally, it holds that
\begin{align*}
\int_{0}^{\infty}g(x)dx = \frac{1}{k^3}\,\Gamma(m+2)\,,
\end{align*}
and 
\begin{align*}
	\int_{0}^{\infty}|g^{\prime}(x)|dx = \frac{1}{k^2}\int_{0}^{\infty}y^{m}e^{-y}\left|(m+1)-y\right|dy\,,
\end{align*}
then inequality \eqref{eq: Euler_sum_MacMahon_ineq} gives that there exists a constant $C_{m} > 0$, not depending on $k$, such that 
		\begin{align*}
	\left|\sum_{k=1}^{\infty}s\sum_{j=1}^{\infty}(sj)^{m+1}k^{m-1}{e^{-(js)k}}-\zeta(3)\Gamma(m+2)\right| \leq s \, C_m \, \zeta(2)\,,
\end{align*}
and therefore we conclude that 
\begin{align*}
\lim_{s \downarrow 0}s^{m+2}F^{(m)}(-s) = \zeta(3)\Gamma(m+2)\,.
\end{align*}
\end{proof}

A direct application of Proposition \ref{propo: asymptotic_MacMahon_moments} gives that for any $m \geq 3$ there exists a constant $C_{m}>0$ such that 
\begin{align}\label{eq: asymptotic_quotient_fulcrums_G_p}
	\frac{F^{(m)}(-s)}{F^{\prime\prime}(-s)^{m/2}} \sim C_{m} \cdot s^{m-2}\,, \quad \text{ as } s\downarrow0\,,
\end{align}
then, applying Theorem \ref{thm: gaussianity all derivatives}, we conclude that the OGF of the plane partitions is Gaussian.

\subsubsection{OGF of the partitions with parts in a set $A$ of positive density.} Here we prove that the integer moments of the normalized Khinchin family associated to the OGF of the partitions with parts in a set $A$ of positive density converge to the integer moments of a standard Gaussian. 
\medskip

Denote by
\begin{align*}
d(A) = \lim_{n \rightarrow \infty}\frac{\# (A\cap[1,n])}{n}\,,
\end{align*}
the natural density of the set $A$.
\medskip

Let $A \subset \{1,2,3,...\}$ be a subset of natural numbers of positive density $d(A)>0$. We denote by
\begin{align*}
P_{A}(z) = \prod_{j \in A}^{\infty}\frac{1}{1-z^{j}}\,, \quad \text{ for any } |z|<1\,,
\end{align*}
the OGF of the partitions of integers with parts in the set $A$.
\medskip

The fulcrum of the power series $P_A(z)$ is given by
\begin{align*}
F_{A}(s) = -\sum_{j \in A}\ln(1-e^{sj}) = \sum_{k = 1}^{\infty}k^{m-1}\sum_{j\in A}j^me^{(js)k}
\end{align*}
and therefore, for any $m \geq 0$, we have that
\begin{align*}
F_A^{(m)}(s)= \sum_{k = 1}^{\infty}k^{m-1}\sum_{j \in A}j^me^{(js)k}
\end{align*}

\begin{propo}\label{lem: asymptotic_density_moments} For any integer $m \geq 0$ we have that
\begin{align*}
F_{A}^{(m)}(-s) \sim d(A)\zeta(2)\Gamma(m+1)\frac{1}{s^{m+1}}\,, \quad \text{ as } s \downarrow 0\,.
\end{align*}	 
\end{propo}
	\begin{proof}  Using that 
	\begin{align*}
		\lim_{j \rightarrow +\infty}\frac{l_j}{j} = \frac{1}{d(A)}\,,
	\end{align*}
	we obtain that, given $0 < \varepsilon <1$, there exists an integer $N>0$ such that 
	\begin{align}\label{eq: ineq_lj}
		\frac{(1-\varepsilon)}{d(A)}j \leq l_j \leq   \frac{(1+\varepsilon)}{d(A)}j\,, \quad \text{ for any } j > N
	\end{align} 
	
	We have the equality
	\begin{align*}
		s^{m+1}F_A^{(m)}(-s) &= \sum_{k = 1}^{\infty}k^{m-1}s\sum_{j =1}^{\infty}(sl_j)^me^{-(l_j s)k}\,. \\
	\end{align*}
We divide the previous sum into two pieces
	\begin{align}\label{eq: sum_pieces}
	\begin{split}
		\sum_{k = 1}^{\infty}k^{m-1}s\sum_{j =1}^{\infty}(sl_j)^me^{-(l_j s)k} 
		&=\sum_{k = 1}^{\infty}k^{m-1}s\sum_{j \leq N}(sl_j)^me^{-(l_j s)k}  \\ &+ \sum_{k = 1}^{\infty}k^{m-1}s\sum_{j > N}(sl_j)^me^{-(l_j s)k}.
	\end{split}
	\end{align}
	\medskip
	
	First we study the term
	\begin{align*}
		\sum_{k = 1}^{\infty}k^{m-1}s\sum_{j \leq N}(sl_j)^me^{-(l_j s)k}\,.
	\end{align*}
	Fix $m = 0$, we have 
	\begin{align*}
		\sum_{j \leq N}s\sum_{k = 1}^{\infty}\frac{e^{-l_j s k}}{k} = -\sum_{j \leq N}s\ln(1-e^{-sl_j})\,,
	\end{align*}
	therefore
	\begin{align*}
	\lim_{s \downarrow 0}\sum_{j \leq N}s\sum_{k = 1}^{\infty}\frac{e^{-l_j s k}}{k}  = 0\,.
	\end{align*}
	Now fix $m \geq 1$. Observe that
	\begin{align*}
		\sum_{k = 1}^{\infty}k^{m-1}s\sum_{j \leq N}(sl_j)^me^{-(l_j s)k} =s\sum_{j \leq N}l_j^m s\sum_{k = 1}^{\infty}(sk)^{m-1}e^{-(l_j s)k} 
	\end{align*}
	therefore 
	\begin{align*}
		\sum_{k = 1}^{\infty}k^{m-1}s\sum_{j \leq N}(sl_j)^me^{-(l_j s)k} \leq s\sum_{j \leq N}l_j^m s\sum_{k = 1}^{\infty}(sk)^{m-1}e^{-sk}
	\end{align*}
	Lemma \ref{lemma: Euler_summation} gives that 
	\begin{align*}
		\lim_{s \downarrow 0}\sum_{j \leq N}l_j^m s\sum_{k = 1}^{\infty}(sk)^{m-1}e^{-sk} = \left(\sum_{j \leq N}l_j^m\right) \Gamma(m)\,,
	\end{align*}
	then 
	\begin{align}\label{eq: finite_sum_lj_zero}
		\lim_{s \downarrow 0}\sum_{k = 1}^{\infty}k^{m-1}s\sum_{j \leq N}(sl_j)^me^{-(l_j s)k} = 0\,.
	\end{align}

	Now we study the second term in \eqref{eq: sum_pieces}, that is, 
	\begin{align*}
		\sum_{k = 1}^{\infty}k^{m-1}s\sum_{j > N}(sl_j)^me^{-(l_j s)k}.
	\end{align*}
	Fix $0<\varepsilon<1$. Using inequality \eqref{eq: ineq_lj} we find that
	\begin{align}\label{eq: inequality_right_limit_density}
	\begin{split}
		\frac{(1-\varepsilon)^m}{d(A)^m}s\sum_{j > N}(sj)^me^{-(js)k(1+\varepsilon)/d(A)} &\leq s\sum_{j > N}(sl_j)^me^{-(l_j s)k}  \\ &\leq \frac{(1+\varepsilon)^m}{d(A)^m}s\sum_{j > N}(sj)^me^{-(js)k(1-\varepsilon)/d(A)}\,,
	\end{split}
	\end{align}
	therefore applying Lemma \ref{lemma: Euler_summation} to the right-hand side of inequality \eqref{eq: inequality_right_limit_density}, and dominated convergence, we conclude that 
	\begin{align*}
		\lim_{s \downarrow 0}s^{m+1}F_{A}^{(m)}(-s) \leq \frac{(1+\varepsilon)^m}{(1-\varepsilon)^{m+1}}d(A)\zeta(2)\Gamma(m+1)\,.
	\end{align*}

	Now we prove, using almost the same argument, that there is a similar lower bound for our limit. We have the inequality
	\begin{align*}
		\frac{(1-\varepsilon)^m}{d(A)^m}s\sum_{j > N}(sj)^me^{-(js)k(1+\varepsilon)/d(A)} \leq s\sum_{j > N}(sl_j)^me^{-(l_j s)k}
	\end{align*}
	and asymptotically we have 
	\begin{align*}
		\frac{(1-\varepsilon)^m}{(1+\varepsilon)^{m+1}}d(A)\zeta(2)\Gamma(m+1)\leq \lim_{s \downarrow 0}s^{m+1}F_A^{(m)}(-s)
	\end{align*}
	Here we use that  
	\begin{align*}
		\lim_{s \downarrow 0}\frac{(1-\varepsilon)^m}{d(A)^m}\sum_{k = 1}^{\infty}sk^{m-1}\sum_{j \leq N}(sj)^me^{-(js)k(1+\varepsilon)/d(A)} = 0.
	\end{align*}
We proved the inequality
	\begin{align*}
			\frac{(1-\varepsilon)^m}{(1+\varepsilon)^{m+1}}d(A)\zeta(2)\Gamma(m+1) &\leq \lim_{s \downarrow 0}s^{m+1}F_A^{(m)}(-s) \\ &\leq \frac{(1+\varepsilon)^m}{(1-\varepsilon)^{m+1}}d(A)\zeta(2)\Gamma(m+1)\,.
	\end{align*}
		Making $\varepsilon \rightarrow 0$ we conclude that for any $m \geq 0$ we have
	\begin{align*}
		\lim_{s \downarrow 0}s^{m+1}F_A^{(m)}(-s) = d(A)\zeta(2)\Gamma(m+1)\,.
	\end{align*}
This concludes the proof. 
\end{proof}

Applying Proposition \ref{lem: asymptotic_density_moments} we conclude that for any $m \geq 3$ there exists a constant $C_{A.m}>0$ such that 
\begin{align*}
\frac{F_{A}^{(m)}(-s)}{F_{A}^{\prime\prime}(-s)^{m/2}} \sim C_{A,m} \, s^{m/2-1}\,, \quad \text{ as } s \downarrow 0\,,
\end{align*}
therefore, applying Theorem \ref{thm: gaussianity all derivatives}, we conclude that the OGF of the partitions of integers with parts in a set $A$ of positive density $d(A)>0$ is Gaussian. 
This claim includes various types of integer partitions, one of which is the partitions of integers where the parts are in arithmetic progression.

\subsection{Other applications}

We now apply the Gaussianity criterion to a specific family of canonical products, to the exponential of entire functions of finite order, and to the exponentials of clans.

\subsubsection{Canonical products of genus $0$ with negative zeros in the class $\K$.}
\medskip

As an example of the use of the gaussianity criterium of Theorem \ref{thm: gaussianity all derivatives} we next consider the power series in the class $\mathcal{Q}$. These power series are entire functions in $\K$ of genus 0 whose only zeros are negative. For convenience we require the normalization $f(0)=1$. Because of Hadamard's factorization Theorem, such an $f$ is a canonical product and may be expressed as 
\begin{align*}
	(\star) \quad f(z) = \prod_{j = 1}^{\infty}\left(1+\frac{z}{b_j}\right), \quad \text{ for any } z \in \C.
\end{align*}
where $(b_k)_{k \geq 1}$ is an increasing sequence of positive numbers such that
\begin{align*}
	\sum_{j = 1}^{\infty}\frac{1}{b_j}<+\infty\,.
\end{align*}
Observe that 
$$m_f(t)=\sum_{j=1}^\infty \frac{t}{(b_j+t)}\quad \mbox{and} \quad \sigma_f^2(t)=\sum_{j=1}^\infty \frac{t b_j}{(b_j+t)^2}\, , \quad \mbox{for any $t\ge 0$}\,,$$ 
and also that  
$\sigma_{f}^2(t)<m_f(t)\, ,  \mbox{ for any $t >0$}$, see \cite{CFFM3} for further details. 
\medskip

The fulcrum $F$ of  $f$ is given by 
\begin{align*}
	F(s) =  \ln(f(e^s)) =\sum_{j = 1}^{\infty}\ln\left(1+\frac{e^s}{b_j}\right)\, , \quad \mbox{for any $s \in \R$}\,.
\end{align*}

\begin{lem}\label{lemma: ineq_canonical}
	For each integer $k \ge 2$, there exists a constant $C_k >0$ such that for any $f$ as in $(\star)$ and any $s\in \R$
	\begin{equation*}
	\label{eq: ineq_canonical}\big|F^{(k)}(s)\big|\le C_k F^{\prime\prime}(s)\, .
	\end{equation*}
\end{lem}
\begin{proof}
	
	The first two derivatives of $F$ are given by
	\medskip
	\begin{align*}
		F^{\prime}(s) =  \sum_{j = 1}^{\infty}\frac{e^s}{(b_j+e^s)}\,, \quad \text{ and } \quad
		F^{\prime\prime}(s) =  \sum_{j = 1}^{\infty}\frac{b_je^{s}}{(b_j+e^s)^2}\,, \quad \mbox{for any $s \in \R$}\,.
	\end{align*}
	\medskip
	
	For any integers $n \geq 1$ and $b >0$, and any $s \in \R$, denote
	\vspace{.2 cm}
	\
	\begin{align*}
		A_{n}(b,s) = \frac{b e^{ns}}{(b+e^{s})^{n+1}}\,.
	\end{align*}
	\medskip
	
	For any $n \geq 1, b >0$ and $s \in \R$ we have the inequality
	\begin{align*}
		A_{n}(b,s) = \frac{be^{ns}}{(b+e^{s})^{n+1}}  = \left(\frac{e^{s}}{b+e^s}\right)^{n-1}\frac{be^{s}}{(b+e^s)^2} \leq A_{1}(b,s)\,.
	\end{align*}
	and also the differential identity
	\begin{align*}
		\frac{d}{ds}A_{n}(b,s) = nA_{n}(b,s)-(n+1)A_{n+1}(b,s)\,.
	\end{align*}
	\medskip
	
	By iteration, we obtain, from the identity above, and for any $k\ge 2$, real coefficients $\delta_{k,l}$, for $k \ge l \ge 0$, so that
	$$\frac{d^k}{ds^k}A_{1}(b,s)=\sum_{l=0}^k \delta_{k,l} A_{1+l}(b,s)\,, \quad \mbox{for any $s \in \R$ and any $b >0$}\,,$$
	and thus    that
	$$\Big|\frac{d^k}{ds^k}A_{1}(b,s)\Big|\le \sum_{l=0}^k |\delta_{k,l}| A_{1+l}(b,s) \le \left(\sum_{l=0}^k |\delta_{k,l}|\right) \, A_{1}(b,s) >0, \quad \mbox{for any $s \in \R$ and any $b>0$}\,.$$

	For $k \ge 2$, since $F^{(k)}(s)=\sum_{j=1}^\infty A_{1}^{(k-2)}(b_j,s)$, we conclude that, for any $s \in \R$, we have the inequality
	\begin{align*}
	|F^{(k)}(s)|\le \sum_{j=1}^\infty \Big|\frac{d^{(k{-}2)}}{ds} A_{1}(b_j,s)\Big|\le \underset{C_k}{\underbrace{\left(\sum_{l=0}^{k{-}2} |\delta_{k-2,l}|\right)}} \, \sum_{j=1}^\infty A_{1}(b_j,s)=C_k F^{\prime\prime}(s)\,.
	\end{align*}
\end{proof}

We use Lemma \ref{lemma: ineq_canonical} and Theorem \ref{thm: gaussianity all derivatives} to  prove

\begin{propo}\label{thm: canonical_gaussian} Let $f$ be a power series in $\mathcal{Q}$. If  $\lim_{t \to +\infty}\sigma_f^2(t) = +\infty$, then $f$ is Gaussian.
\end{propo}
\begin{proof} We use Theorem \ref{thm: gaussianity all derivatives}.
	We are going to show that for any $k \ge 3$ we have
	$$\lim_{s \to +\infty} \frac{|F^{(k)}(s)|}{F^{\prime\prime}(s)^{k/2}}=0\,.$$
	But 
	from Lemma \ref{lemma: ineq_canonical} above 
	\smallskip
	$$\frac{|F^{(k)}(s)|}{F^{\prime\prime}(s)^{k/2}}=\frac{|F^{(k)}(s)|}{F^{\prime\prime}(s)} \cdot {F^{\prime\prime}(s)}^{1-k/2}\le C_k \sigma_f(e^s)^{2-k},$$
	which tends to 0, as $s \to +\infty$ since $\lim_{s \to+\infty} \sigma_f(e^s)=+\infty$ and $k \ge 3$.
\end{proof}

Now we give examples of canonical products $f \in \mathcal{Q}$ of any order $\rho \in (0,1)$ and such that $\lim_{t\rightarrow+\infty}\sigma_f^2(t) = +\infty$.

\begin{example}[Gaussian canonical products in $\mathcal{Q}$ of order $\rho \in (0,1)$] Fix $\rho \in (0,1)$, if we take $b_n = n^{1/\rho}$ and define the canonical product $f \in \mathcal{Q}$ given by
\begin{align*}
f(z) = \prod_{j=1}^{\infty}\left(1+\frac{z}{b_j}\right)\,, \quad \text{ for any } z \in \C\,,
\end{align*}
then $\lim_{t \rightarrow+\infty}\sigma_f^2(t) =+\infty$, this follows from Proposition 5.4 in \cite[p. 32]{CFFM3}, and therefore $f$ is Gaussian.

\end{example}

The following result implies that for canonical products in $\mathcal{Q}$ of order $\rho \in (0,1)$, if the limit
\[
\lim_{s \to +\infty} \frac{F^{(k)}(s)}{F^{\prime\prime}(s)^{k/2}}
\]
exists for any $k \geq 3$, then it must equal zero.

\begin{propo} Let $f \in \mathcal{Q}$ be a canonical product. Assume that $f$ has order $\rho \in (0,1)$, then
	\begin{align*}
		\liminf_{s\rightarrow+\infty}\frac{|F^{(k)}(s)|}{F^{\prime\prime}(s)^{k/2}} = 0\,, \quad \text{ for any } k \geq 3.
	\end{align*}	
\end{propo}
\begin{proof}
	For entire functions $f \in \K$ of order $\rho>0$ we always have 
	\begin{align*}
		\limsup_{t \rightarrow \infty}\sigma_f^2(t) = +\infty\,,
	\end{align*}	
	otherwise, upon integration, we would find that $\rho = 0$; for an analysis of the behavior of the variance function see \cite[p.12]{CFFM3}. Combining this fact with Lemma \ref{lemma: ineq_canonical} we obtain that for canonical products $f \in \mathcal{Q}$ of order $\rho \in (0,1)$ we always have
	\begin{align*}
		\liminf_{s\rightarrow+\infty}\frac{|F^{(k)}(s)|}{F^{\prime\prime}(s)^{k/2}} = 0\,, \quad \text{ for any } k \geq 3\,.
	\end{align*}	
\end{proof}

\subsubsection{Exponential of an entire function of finite order in $\K$.} As another application of our moment criterium we prove that the exponential of a non-constant entire function of finite order, having non-negative coefficients, verify that the integer moments of the family $(\breve{X}_t)$ converge, as $t\rightarrow\infty$, to the integer moments of a standard normal random variable. As a consequence of this convergence we obtain, by a direct application of Theorem \ref{thm: gaussianity all derivatives}, that $f=e^g$ is Gaussian. 
\medskip

We have the following, standard, upper bound for the quotient of consecutive derivatives of a transcendental entire function.
\begin{lem}\label{lemma: entirefinitorderbound} Let $g$ be a transcendental entire function of finite order with non-negative coefficients, then there are constants $C>0$, $M>0$ and $T>0$ such that 
\begin{align*}
\frac{g^{\prime}(t)}{g(t)}\leq Ct^{M}, \quad \text{ for any } t\geq T\,.
\end{align*}
\end{lem}
\begin{proof}
Using that $m_f$ is increasing, for any $\lambda>1$, we have
\begin{align*}
	m_g(t)\ln(\lambda) \leq \int_{t}^{\lambda t}m_g(s)\frac{ds}{s} = \ln(g(\lambda t)) - \ln(g(t)) 
\end{align*}

The function $g$ is a non-constant entire function with non-negative coefficients (increasing and unbounded), then there exists some $M>0$ such that $g(t) >1$, for any $t > M$, otherwise $g$ would be constant. From the previous inequality we obtain that $-\ln(g(t))<0$, for any $t > M$, and therefore 
\begin{align*}
	m_g(t)\ln(\lambda) \leq \ln(g(\lambda t)), \quad \text{ for any } t>M.
\end{align*}

Now, given that $g$ has finite order $\rho\geq0$: for any fixed $\lambda>1$, and for any $\varepsilon>0$ there are constants $C>0$ and $T>0$ such that 
\begin{align*}
	m_g(t) \leq Ct^{\rho+\varepsilon}\,,\quad \text{ for any } t>T,
\end{align*}
which is equivalent to 
\begin{align}\label{eq: inequality_finite_order}
	\frac{g^{\prime}(t)}{g(t)} \leq Ct^{\rho+\varepsilon-1}\,,\quad \text{ for any } t>T\,.
\end{align}	
The result follows from the previous inequality. 
\end{proof}

Let $g$ be a non-constant power series with radius $R>0$ and denote $f=e^g$ its exponential, then the fulcrum of $f$ can be written as
\begin{align*}
F(z) = g(e^z)\,,\quad \text{ for any } z \in \D(0,R).
\end{align*}

Our first lemma expresses the derivatives of the fulcrum $F$ as finite linear combination of terms of the form $e^{js}g^{(j)}(e^s)$, where $j \geq 1$.
\begin{lem}\label{lemma: derivatives_fulcrum_linear_combination} Let $g$ be a non-constant power series with non-negative coefficients and radius $R>0$. Denote $f =e^g$ is exponential. Then, for any $k \geq 1$, there are integer constants $c_1,\dots,c_{k-1} \geq 0$ such that 
	\begin{align*}
		F^{(k)}(s) =e^{ks}g^{(k)}(e^s)+\sum_{j = 1}^{k-1}c_j (e^{sj} g^{(j)}(e^s)), \quad \text{ for any } s < \ln(R)\,.
	\end{align*}	
\end{lem}
%

From the previous lemma we obtain asymptotic comparisons, in terms of $g$, for the fulcrum $F$ and any of its derivatives.
\begin{lem}\label{coro: lemma: derivativa_fulcrum_e^g_comparable} Let $g$ be a non-constant power series with non-negative coefficients and radius $R>0$. Denote $f=e^g$ its exponential, then: 
	\begin{itemize}	
		\renewcommand{\labelitemi}{\raisebox{0.3ex}{\scalebox{0.6}{$\blacksquare$}}}
		\item If $g$ is not a polynomial, then
		\begin{align*}
			F^{(k)}(s) \asymp e^{ks}g^{(k)}(e^s), \quad \text{ as } s \uparrow \ln(R).
		\end{align*}
		\item If $g(z) = a_Nz^n+\dots+a_0$ is a polynomial of degree $N \geq 1$, then, for any $k \geq 1$, we have 
		\begin{align*}
			F^{(k)}(s) \asymp a_Ne^{sN}, \quad \text{ as } s \rightarrow +\infty.
		\end{align*}
	\end{itemize}
\end{lem}
\medskip

\begin{propo}\label{thm: converge_moments_e^g_transcedental}Let $g$ be a non-constant entire function, with non-negative coefficients, and finite order $\rho\geq0$. Let $f=e^g$ be its exponential and denote $(\breve{X}_t)$ its normalized Khinchin family, then
\begin{align*}
\lim_{t\rightarrow\infty}\E(\breve{X}_t^m) = \E(Z^m)\,, \quad \text{ for any integer } m\geq1\,.
\end{align*}
Here $Z$ denotes a standard normal random variable. 
\end{propo}
\begin{proof} First assume that $g(z)=a_Nz^N+a_{N-1}z^{N-1}+\dots+a_0$ is a polynomial of degree $N$, then by virtue of Lemma \ref{coro: lemma: derivativa_fulcrum_e^g_comparable} we have 
	\begin{align*}
		\frac{F^{(k)}(s)}{F^{\prime\prime}(s)^{k/2}} \asymp \frac{1}{(a_Ne^{Ns})^{k/2-1}}\,, \quad \text{ as } s\rightarrow+\infty\,.
	\end{align*}
and therefore 
\begin{align*}
\lim_{s\rightarrow+\infty}\frac{F^{(k)}(s)}{F^{\prime\prime}(s)^{k/2}} = 0\,, \quad \text{ for any } k \geq 3\,.
\end{align*}
	
Assume now that $g$ is a transcendental entire function of finite order $\rho\geq0$. For any $k\geq 3$, applying Lemma \ref{coro: lemma: derivativa_fulcrum_e^g_comparable}, we conclude that
\begin{align*}
\frac{F^{(k)}(s)}{F^{\prime\prime}(s)^{k/2}} \asymp \frac{g^{(k)}(e^s)}{g^{\prime\prime}(e^s)^{k/2}}\,, \quad \text{ as } s\rightarrow+\infty\,.
\end{align*}

Because $g$ is a transcendental entire function Lemma \ref{lemma: entirefinitorderbound} gives that: for any $k \geq 3$, there are constants $C>0$, $M>0$ and $T>0$ such that
\begin{align*}
\frac{g^{(k)}(e^s)}{g^{\prime\prime}(e^s)^{k/2}} = \frac{g^{(k)}(e^s)}{g^{(k-1)}(e^s)}\dots\frac{g^{\prime\prime\prime}(e^s)}{g^{\prime\prime}(e^s)}\frac{1}{g^{\prime\prime}(e^s)^{k/2-1}} \leq C\frac{e^{Ms}}{g^{\prime\prime}(e^s)^{k/2-1}}\,, \quad \text{ for any } t\geq T\,.
\end{align*}
\smallskip

Using again that $k\geq 3$ and $g$ is transcendental we conclude that 
\begin{align*}
\lim_{s\rightarrow+\infty}\frac{e^{Ms}}{g^{\prime\prime}(e^s)^{k/2-1}} = 0\,,
\end{align*}
and therefore 
\begin{align*}
\lim_{s\rightarrow+\infty}\frac{F^{(k)}(s)}{F^{\prime\prime}(s)^{k/2}} = 0\,, \quad \text{ for any } k \geq 3\,.
\end{align*}
\end{proof}

	\begin{coro}\label{coro: exps_entire_Gaussian} Assume that $g$ is a non-constant entire function, with non-negative coefficients, and finite order $\rho\geq0$, then $f = e^g$ is Gaussian. 
\end{coro}
\begin{proof} This result follows by combining Theorems \ref{thm: gaussianity all derivatives} and Proposition \ref{thm: converge_moments_e^g_transcedental}.
\end{proof}


\subsubsection{Exponentials of clans.} The exponential of a clan is always a Gaussian power series. In fact: all the integer moments of the normalized family $(\breve{X}_t)$ converge to the moments of a standard normal, as $t\uparrow R$. This situation covers some of the cases $f=e^g$, with $g \in \K$ an entire function of infinite order. 
\medskip

We recall here the definition of clan: we say that $f \in \K$ is a clan if 
\begin{align*}
\lim_{t\uparrow R}\frac{\sigma_f(t)}{m_f(t)} = 0\,,
\end{align*}
clans concentrate around the mean, that is, for clans we have $X_t/m_f(t) \stackrel{\P}{\rightarrow} 1$, as $t \uparrow R$ (convergence in probability). See \cite{CFFM3} for further details; a comprehensive analysis of clans is provided there.

\begin{lemmaletra}\label{lemma: clans_asymptotic} Let $g$ be a non-constant power series with non-negative coefficients and radius $R>0$. Assume that $g$ is a clan, then 
 \begin{align*}
 g^{(k)}(t) \sim \frac{g^{\prime}(t)^{k}}{g(t)^{k-1}}\,, \quad \text{ as } t\uparrow R\,.
 \end{align*}
\end{lemmaletra}
This is a direct consequence of Corollary 3.5 in \cite{CFFM3}. 

\begin{theo}\label{thm: exponentialclanGaussian} Let $g$ be a non-constant power series with non-negative coefficients and radius $R>0$ and denote $f = e^g$ its exponential. Assume that $g$ is a clan and also that $\lim_{t\uparrow R}g(t)=+\infty$, then 
\begin{align*}
\lim_{s\uparrow\ln(R)}\frac{F^{(k)}(s)}{F^{\prime\prime}(s)^{k/2}} = 0\,, \quad \text{ for any } k\geq3\,,
\end{align*}
the integer moments of the family $\breve{X}_t$ converge, as $t \uparrow R$, to the integer moments of a standard normal random variable and $f$ is Gaussian. 
\end{theo}
\begin{proof} For polynomials we have already proved this claim, see Corollary \ref{coro: exps_entire_Gaussian}. Assume that $g$ is not a polynomial, using the asymptotic comparisons given by Lemma \ref{coro: lemma: derivativa_fulcrum_e^g_comparable} we find that
\begin{align*}
\frac{F^{(k)}(s)}{F^{\prime\prime}(s)^{k/2}} \asymp \frac{g^{(k)}(e^s)}{g^{\prime\prime}(e^s)^{k/2}}\,,\quad \text{ as } s\uparrow\ln(R).
\end{align*} 
Now Lemma \ref{lemma: clans_asymptotic} gives that 
\begin{align*}
\frac{g^{(k)}(e^s)}{g^{\prime\prime}(e^s)^{k/2}} \sim \frac{1}{g(e^s)^{k/2-1}}\,, \quad \text{ as } s\uparrow\ln(R).
\end{align*}
\\
We conclude the proof by using that $k\geq 3$ and $\lim_{t\uparrow R}g(t) = +\infty$. 
\end{proof}

\begin{example} Denote by $f(z) = e^z$ the exponential. For any $k \geq 1$ the $k$-th iterate of the exponential function
\begin{align*}
f_k(z) = f\circ \stackrel{k}{\dots} \circ f(z)
\end{align*}
is a Gaussian power series.
\medskip

Indeed: we prove first that the exponential $f(z)=e^z$ conforms a clan; we have that
\begin{align*}
	\lim_{t\rightarrow+\infty}\frac{\sigma_f(t)}{m_f(t)} = \lim_{t\rightarrow+\infty}\frac{\sqrt{t}}{t} = 0\,,
\end{align*}
see also \textnormal{\cite{CFFM3}[p.16].}
\medskip

For $k = 1$ observe that $f_1(z) = f(z) = e^z$. We already know that the exponential function is Gaussian; this follows, for instance, applying Theorem \ref{thm: gaussianity all derivatives}. See also \textnormal{\cite[p. 863]{CFFM1}}, where a direct proof of this fact is given.
\medskip

Now using that the exponential of a clan is always a clan, see \cite[pp. 19-20]{CFFM3}, we obtain that for any $k \geq 1$ the power series
\begin{align*}
f_k(z) = f\circ \stackrel{k}{\dots} \circ f(z)
\end{align*}
is a clan. Applying Theorem \ref{thm: exponentialclanGaussian}, we conclude that for any $k\geq1$ the power series $f_k(z)$ is Gaussian (for $k = 1$ the function $f_1(z) = e^z$ is Gaussian and for any $k \geq 2$ the power series $f_k(z)$ is the exponential of a clan)
\medskip

In fact: let $g$ be a non-constant power series with non-negative coefficients, which is not a polynomial, and has radius of convergence $R>0$. Assume that $g$ is a clan and also that $\lim_{t\uparrow R}g(t) = +\infty$, then Theorem \ref{thm: exponentialclanGaussian} gives that for any $k \geq 1$ the power series
\begin{align*}
	f_k(g(z)) = f\circ \stackrel{k}{\dots} \circ f(g(z))\,,
\end{align*}
is Gaussian. Here we use, iteratively, that the exponential of a clan is again a clan, see \cite[pp. 19-20]{CFFM3}, and also that any iterate of the exponential function conforms a clan.

\end{example}

\noindent\textbf{Funding.}  Research of V.\,J. Maci\'a was partially funded by grants MTM2017-85934-C3-2-P2
and PID2021-124195NB-C32 of Ministerio de Econom\'{\i}a y Competitividad of Spain, by the European Research Council Advanced Grant 834728,
by the Madrid Government (Comunidad de Madrid-Spain) under programme PRICIT, and by Grant UAM-Santander for the mobility of
young researchers 2023.

\end{document}